\documentclass[11pt, oneside]{article}  
\pdfoutput=1 

\usepackage{geometry}                		
\usepackage{graphicx}				
\usepackage{amssymb}
\usepackage{amsmath, amsthm, amssymb, amsfonts}
\usepackage{epstopdf}
\usepackage{algorithm}
\usepackage{algpseudocode}

\newcommand{\Argmin}[1]{{\operatorname{argmin}}_{#1}\;}
\newcommand{\argmax}[1]{\underset{#1}{\operatorname{argmax}}\;}
\newcommand{\argmin}[1]{\underset{#1}{\operatorname{argmin}}\;}

\newtheorem{thm}{Theorem}

\newtheorem{prop}[thm]{Proposition}
\long\def\Lcal{\mathcal{L}}

\title{Half-quadratic transportation problems}
\author{Mariano Rivera$^{1,2}$ \\ \\
$^{1}$Centro de Investigacion en Matematicas AC, Guanajuato, Gto. 36000, Mexico \\
$^{2}$ Instituto Potosino de Ciencia y Tencnologia AC, San Luis Potosi, SLP. 78216, Mexico}

\date{}                                          

\begin{document}
\maketitle

\begin{abstract}

We present a primal--dual memory efficient algorithm for solving a relaxed version of the general transportation problem. Our approach approximates the original cost function with a differentiable one that is solved as a sequence of weighted quadratic transportation problems. The new formulation allows us to solve differentiable, non--convex transportation problems.


\end{abstract}



\section{Introduction}

The general transportation problem (GTP) deals with the distribution of goods from $m$ suppliers with production capacities $p = \{ p_i\}_{i=1,\ldots,m}$ to $n$ destinations with demands $q = \{ q_i\}_{i=1,\ldots,n}$.  Without loss of generality, we assume balanced production and demand: $\sum_i p_i = \sum_j q_j$. A classical approach to this problem assumes that the cost of transport remains constant, independently of the the quantity to be transported. In real problems, this is not the case. The cost may increase or decrease according the volume of the transported good. We can write the general transportation problem as follows:
\begin{eqnarray}
	\textstyle \min_x && \textstyle\sum_{i,j} f_{ij}(x_{ij}) \label {eq:gtp1}\\
	\mbox{s.t.}  && \textstyle\sum_j  x_{ij}  = p_i \label {eq:gtp2}\\
	&& \textstyle \sum_i  x_{ij} = q_j    \label {eq:gtp3} \\
	&& x_{ij} \ge 0 \label {eq:gtp4}.
\end{eqnarray}
where $x_{ij}$ denotes the quantity of goods to be transported from the $i$th supplier to the $j$th destination, $f: [ m \times n ] \times \mathbb{R^+}\rightarrow \mathbb{R^+}$ is a continuous cost function that depends on the supplier $i$, the  destination $j$ and the volume $x_{ij}$. 

The first case of a transportation problem was formulated by Hitchcock assuming that the cost functions are linear: $f_{ij}(x_{ij}) = c_{ij}x_{ij}$ \cite{Hitchcock41}. Another popular model is the quadratic one: $ f_{ij}(x_{ij}) = a_{ij}x_{ij}^2 + b_{ij}x_{ij}$.  According to Ref. \cite{adlakha13}, the quadratic model is popular because it can approximate other cost functions. Despite such flexibility, the limitation of quadratic models has been well documented in the context of robust statistics, and its applications to image processing and computer vision \cite{gemanyang95,charbonnier97}. In our opinion, the main limitations of quadratic models are the following:

\begin{enumerate}

\item The impossibility of limiting the cost for large values;  \emph{i.e.},  one has $\lim_{x \rightarrow +\infty} |q_{ij}(x_{ij}) | = \infty$. 

\item The limitation of promoting sparse solutions; \emph{i.e.}, solutions that use a reduced number of routes.

\end{enumerate}

In this work, we present an approximation scheme that allows us to define new cost functions that overcome the aforementioned limitations. We also present a primal--dual algorithm with limited memory requirements.
 
The GTP is relevant in modern computer science applications such as computer vision \cite{rubner:EMD00}, machine learning \cite{baccianella13:emd} and data analysis \cite{levina:emdmallows01,kusner:emd15}. The Earth Mover Distance (EMD) is an interesting application of the transportation problem where the optimum cost is used as a metric between the histograms, vectors, $p$ and $q$.  In Ref. \cite{rubner:EMD00}, EMD is used as a metric for image retrieval in the context of computer vision. Recently,  the EMD was proposed as a measure of reconstruction error for non--negative matrix factorisation \cite{zen:emd14}. The Word Moving Distance is the metric version for comparing documents based on the transportation problem \cite{kusner:emd15}. Recent EMD applications include the quantification of biological differences in flow cytometry samples \cite{orlova:emd16}.
In addition, there is current interest in the learning of metrics for particular problems \cite{cuturi14}; since our proposal is parametrised, the parameters involved can be learned.

\section{Preliminaries}

Before presenting our transportation formulation, we review an important result reported in the context of robust statistics and continuous optimisation applied to image processing \cite{charbonnier97}. The purpose of such work was to transform some non--linear cost functions to half--quadratic functions \cite{gemanyang95,charbonnier97}. A half--quadratic function is quadratic in the original variable and convex in a new auxiliary variable, where the minima of the auxiliar variable can be computed with a closed formula. The next proposition resumes the conditions imposed on the cost function $f$ and the transformed half--quadratic function.

\begin{prop} \label{prop:fhq} Let $f: \mathbb{R^+} \rightarrow \mathbb{R^+}$ be a function that fulfils the following conditions:

\begin{enumerate}

\item $f(t) \ge m$ with $f(0) =m$, for $t\ge 0$ and $ m > -\infty$.

\item $f$ is continuously differentiable.

\item $f'(t) \ge 0$.

\item $\lim_{t \rightarrow +\infty }f'(t)/(2t) =0$.

\item $\lim_{t \rightarrow +0^+ }f'(t)/(2t) =M$, $0<M<+\infty$.

\end{enumerate}

Then, 

\begin{enumerate}

\item there exists a strictly convex and decreasing function  $\psi:(0,M]\rightarrow [0, \beta)$, where 
$
\beta = \textstyle \lim_{t\rightarrow+\infty} \left\{ f(t) - t^2 f'(t) /(2t)  \right\}
$ such that 
\begin{equation}
	\label{eq:hqf}
	f(t) = \inf_{0<\omega \le M} \left\{ \omega t^2 + \psi( \omega) \right\};
\end{equation}

\item the solution to $\inf_{0<\omega \le M} \{\omega t^2 + \psi( \omega)\} $
is unique and given by
\begin{equation}
 	\label{eq:omega}
 	\omega^* = f'(t) / (2t) . 
\end{equation}
\end{enumerate}

\end{prop}
\begin{proof}
The proof is presented in Ref. \cite{charbonnier97}. 
\end{proof}

\noindent Observe that our version of the half--quadratic Proposition assumes a non-negativity constraint on the primal variables.

\section{Half--quadratic transportation problem}

\begin{figure}[t]
\centerline{
\includegraphics[width=0.5\linewidth]{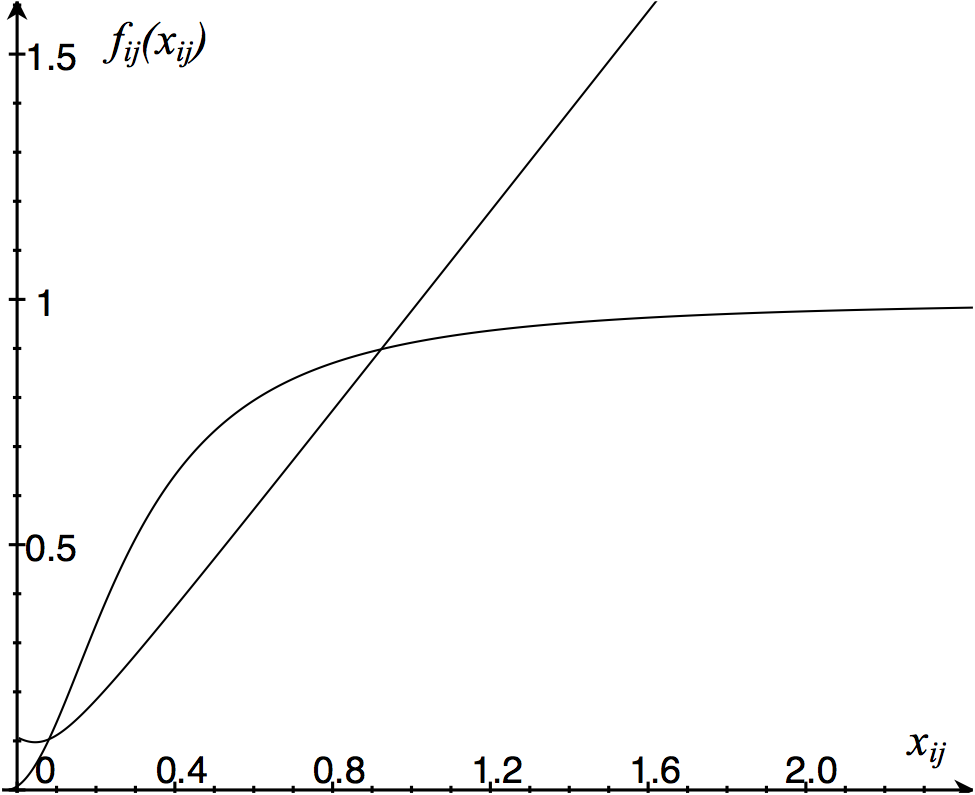}
}
\caption{Plots of half--quadratic cost functions (for $x \ge 0$):  the $L1$--norm approximation $\sqrt{x^2+\beta^2}$ and the $L0$--norm approximation $x^2/ (x^2 + \beta^2)$.}
\label{fig:fig1}
\end{figure}
In this section we present a memory efficient primal--dual algorithm for solving GTP which cost functions satisfy Proposition \ref{prop:fhq}.

\begin{prop} Let $f_{ij}$ be a cost function in \eqref{eq:gtp1}, such that each $f$ satisfies Proposition \ref{prop:fhq}; then, a solution to the transportation problem can be computed with Algorithm \ref{alg:algorithm2}. 
\end{prop}

{\small
\begin{algorithm}[ht]
\caption{Half-quadratic transportation solver.}  \label{alg:algorithm2}
    \begin{algorithmic}[1]
    	\Require {$\sum_i q_i =\sum_j p_j$ and $f,p,q \ge 0;$ }  
	\Function {hqTP} {$f, p,q$}
	\State  Initialise $\lambda = \gamma = 1$, $\omega_{ij} =c_{ij}$ and use $\bar \omega_{i,j} = 1/ \omega_{ij}$
	\Repeat 
		\Repeat
			\State  Update $s$ with \eqref{eq:shq};
			\State Update $\lambda$ with \eqref{eq:lambdahq};
			\State Update $\gamma$ with \eqref{eq:gammahq};
		\Until{convergence}
		\State Update $x$ with \eqref{eq:xhq};
		\State Update $\omega$ with \eqref{eq:kktomega};
	\Until{convergence}
	\State \Return $[x, \lambda, \gamma]$;	
	\EndFunction
        \end{algorithmic}
\end{algorithm}
}

\begin{proof}

\begin{enumerate}

\item On the half quadratic transportation problem. By \eqref{eq:hqf}, the cost  \eqref{eq:gtp1} can be rewritten as 
\begin{equation}
\textstyle \min_{x} \textstyle\sum_{ij} f_{ij}(x_{ij}) = 
\textstyle \min_{x,w} \textstyle\sum_{ij} \{ \omega_{ij} x^2_{ij} + \psi(\omega_{ij}) \}.
\end{equation}

\item On the algorithm convergence. Let $\Lcal$ denotes the Lagrangian of the half--quadratic transportation problem, then one can interchange the order of the minimisations; \emph{i.e.}, 
$
\min_{x,\omega }\max_{\bf y} \Lcal(x,\omega, {\bf y}) = \min_x \min_{\omega} \max_{\bf y}  \Lcal( \cdot )  = \min_\omega \min_{x} \max_{\bf y} \Lcal(\cdot); 
$ 
where we denote with $\bf y$ the Lagrange's multiplies vectors. This suggests an alternating minimisation scheme \emph{w.r.t.} $\omega$ and $(x, {\bf y}$). Let $x^k$, $\omega^{k}$ and ${\bf y}^{k}$ be the current feasible values, then we define 
\begin{equation}
	\label{eq:minomega}
	\omega^{k+1} = \Argmin{\omega} \Lcal( x^k,\omega, {\bf y}^{k} )
\end{equation}
to be the updated $\omega$ value. Thus, $x$ and ${\bf y}$ are updated by solving the quadratic transportation problem: 
\begin{equation}
	\label{eq:minhqtp}
	x^{k+1}, {\bf y}^{k+1} = \argmin{x} \argmax{{\bf y}} \Lcal( x,\omega^{k+1}, {\bf y}).
\end{equation}
We define $ F(x) = \sum_{ij} f_{ij}(x_{ij}) $ and $\hat F(x, \omega) = \sum_{ij} \{ \omega_{ij} x^2_{ij} + \psi(\omega_{ij}) \}$ and observe that $F(x^k) = \hat F(x^k, \omega^k) \ge \hat F(x^k, \omega^{k+1}) \ge \hat F(x^{k+1}, \omega^{k+1})  = F(x^{k+1})$. Then, the alternated minimisations \emph{w.r.t.} $\omega$ and $x$ produce a feasible convergent sequence $\{x^k, x^{k+1}, x^{k+2}, \ldots \}$  that reduces the cost of the GTP: $F(x^k) \ge F(x^{k+1}) \ge F(x^{k+2}) \ge \ldots$. 

\item On the alternated minimisations.  From \eqref{eq:omega}, the optimum $\omega$  in  \eqref{eq:minomega} is computed as
\begin{equation}
	 \omega_{ij} = f'_{ij}(x_{ij})/(2x_{ij}). \label {eq:kktomega}
\end{equation}
for a given $x$.  We define ${\bf y}$ equal to $(\lambda, \gamma, s)$ where $\lambda$ and $\gamma$ are the Lagrange's multipliers for the equality constraints \eqref{eq:kkth2} and \eqref{eq:kkth3}, respectively; and $s$ are the Lagrange's multipliers for the non--negativity constraint.  Then, the minimisation \eqref{eq:minhqtp} corresponds to finding the vectors  $(x, \lambda, \gamma, s)$ that solve the Karush-Kuhn-Tucker conditions (KKTs) with $\omega$ fixed:
\begin{eqnarray}
	 \omega_{ij} x_{ij}  - \lambda_i - \gamma_j  - s_{ij} &=& 0, \label{eq:kkth1}  \\
	\textstyle \sum_j  x_{ij}  - p_i  &=& 0,  \label {eq:kkth2} \\
	\textstyle  \sum_i  x_{ij}  -  q_j &=& 0,   \label {eq:kkth3} \\
	s_{ij} x_{ij} &=&0,  \label {eq:kkth4} \\
	s_{ij}, x_{ij} &\ge&0,  \label {eq:kkth5} 
\end{eqnarray}
A strategy for solving the KKTs is to use an iterative Projected Gauss--Seidel scheme \cite{morales:pgs08}. Thus, from  \eqref{eq:kkth1}:
\begin{equation}
	x_{ ij}  = \bar \omega_{ ij} \left(  \lambda_i + \gamma_j +s_{ ij} \right), \label{eq:xhq} 
\end{equation} 
where we define  $\bar \omega_{ ij} = 1/ \omega_{ ij}$. Substituting $x_{ij}$ in \eqref{eq:kkth2}, we have
\begin{equation}
	  \textstyle \sum_j \bar \omega_{ ij} \left(  \lambda_i + \gamma_j +s_{ ij} \right) = p_i.  \label {eq:sumkkt2} 
\end{equation}
We solve for $\lambda_i$ and obtain
\begin{equation}
	    \lambda_i  = \frac{ p_i - \sum_j (\gamma_j + s_{ij}){\bar \omega_{ ij}} } {\sum_j {\bar \omega_{ ij}} }. \label {eq:lambdahq}
\end{equation}
Similarly, we substitute $x_{ij}$ in \eqref {eq:kkth3} and solve for $\gamma_j$:
\begin{equation}
	   \gamma_j  = \frac{ q_j - \sum_i ( \lambda_i + s_{ij} ){\bar \omega_{ ij}} } {\sum_i {\bar \omega_{ ij}} }. \label {eq:gammahq}
\end{equation}
From \eqref{eq:kkth1}, \eqref{eq:kkth4} and \eqref {eq:kkth5}, we see two cases:  $x_{ij}=0$  and $ s_{ij} = -  ( \lambda_i + \gamma_j ) \ge 0$; or  $x_{ij}=\bar \omega_{ij} ( \lambda_i + \gamma_j ) \ge 0$ and $s_{ij} =0 $. Thus
\begin{equation}
	s_{ ij}  = \max \{0, -  ( \lambda_i + \gamma_j ) \}.
	   \label{eq:shq}
\end{equation}
The complete procedure is shown in Algorithm \ref{alg:algorithm2}. 
\end{enumerate}
\end{proof}

In practice, we observed that the internal loop in Algorithm  \ref{alg:algorithm2}  requires only a few iterarations to approximate the dual variables and is not necessary to achieve convergence; in our experiments we used five iterations. The solution is the global minima or a local one by depending if the cost is or is not convex. Following, we present two particular and interesting cases of half-quadratic transportation problems.  
 
\begin{prop}[Approximation for Linear Transportation]  \label{prop:l1} 
The linear cost can be approximated with the differential function \cite{charbonnier97}:
$
	f_{ij}(ij) =  c_{ij}  ({ x_{ij}^2 + \beta^2} ) ^{\frac{1}{2}}, 
$
with $\beta \approx 0$. Thus, $\omega$ is computed with
$
	 \textstyle \omega_{ij} = {\frac{1}{2}} { c_{ij}} / ({ x_{ij}^2 + \beta^2} ) ^{\frac{1}{2}}
$
 in Algorithm \ref{alg:algorithm2}.
\end{prop}

\begin{proof}
It follows from 
\begin{equation}
	\label{eq:l}
	\textstyle \lim_{\beta \rightarrow 0} {c_{ij}  \left( x_{ij}^2 + \beta^2 \right)^{\frac{1}{2}} }= c_{ij} x_{ij} ; \; c_{ij}, x_{ij} \ge 0.
\end{equation}
Hence, the formula for $\omega$ follows directly from \eqref{eq:kktomega}.
\end{proof}


\begin{prop}[Approximation for $L_0$ Transportation] \label{prop:GM}
The cost function of the form $ f_{ij}(x_{ij}) = c_{ij}[1-\hat \delta(x_{ij})]$, for $c_{ij}, x_{ij} >0$ (where $\hat \delta(t)$ is the Kroneker's delta) can be approximated with the half-quadratic function:
 \begin{equation}
 	 \tilde f_{ij}(x_{ij}) =  c_{ij} x^2_{ij}/ (\beta^2 + x^2_{ij} ) 
 \end{equation}
 with $\beta \approx 0$. Thus, $ \omega_{ij} = { c_{ij} \beta^2 }/{(\beta^2 + x_{ij}^2)^2 }$.
 \end{prop}
\begin{proof} 
It follows from 
\begin{equation}
	\label{eq:l0}
\lim_{\beta\rightarrow 0} \frac{t^2}{\beta^2+t^2}~=~1-\hat \delta(t)~=~\left\{\begin{array}{rcl} 0 &\mathrm{for}& t=0, \\  1 &\mathrm{for}& t\in\mathbb{R}\backslash\{0\}. \end{array}\right.
\end{equation}
The formula for $\omega$ follows directly from \eqref{eq:kktomega}.
\end{proof}

Figure \ref{fig:fig1} plots the half--quadratic functions that approximate the $L_1$ and $L_0$ norms.

\section{Relationship with the quadratic transportation problem}

\begin{prop}[Simple Quadratic Transport (SQT)] \label{prop:sqp} The SQT problem is defined by the cost function $f_{ij}(ij) = c_{ij} x_{ij}^2$; thus, $ \omega_{ij} = c_{ij}$.
\end{prop}
\begin{proof} It follows directly from \eqref{eq:kktomega}.
\end{proof}
\begin{prop}[Simple Quadratic Transportation Algorithm] In the case of the QT problem, the Algorithm  \ref{alg:algorithm2}  is reduced to the dual Algorithm \ref{alg:algorithm1}.
\end{prop}
{\small
\begin{algorithm}[ht]
\caption{Simple quadratic transportation solver.}  \label{alg:algorithm1}
    \begin{algorithmic}[1]
    	\Require {$\sum_i q_i =\sum_j p_j$ and $c,p,q \ge 0;$ }  
	\Function {qTP} {$c, p,q$}
	\State  Initialise $\lambda = \gamma=1$ and use $\bar \omega_{i,j} = 1/ c_{ij}$
	\Repeat 
		\State  Update $s$ with \eqref{eq:shq};
		\State Update $\lambda$ with \eqref{eq:lambdahq};
		\State Update $\gamma$ with \eqref{eq:gammahq};
	\Until{convergence}
	\State Compute $x$ with \eqref{eq:xhq};
	\State \Return $[x, \lambda, \gamma]$;	
	\EndFunction
        \end{algorithmic}
\end{algorithm}
}
\begin{proof} From Proposition \ref{prop:sqp}, we note that $\omega$ is constant, and that the computations of  $\lambda$, $\gamma$ and $s$ are independent of $x$. 
\end{proof}

Dorigo and Tobler discussed the relationship between the QTP and the push--pull migration laws  implemented in Algorithm \ref{alg:algorithm1} \cite{dorigo83}.

\begin{prop}[Quadratic Transportation (QT)] \label{prop:qp} The QT is defined by a cost function of the form $ f_{ij}(ij) = a_{ij} x_{ij}^2 + b_{ij} x_{ij}$. Thus, the dual algorithm is derived with $\omega=c_{ij}$ and using the condition
\begin{equation}
	\omega_{ij} x_{ij}  - \lambda_i - \gamma_j  - s_{ij} = b_{ij} \label{eq:kkthqtp}
\end{equation}
instead of \eqref{eq:kkth1}. 
\end{prop}
\begin{proof} It follows directly from the KKTs.
\end{proof}

\noindent \emph{Remark.} An alternative to Proposition \ref{prop:qp}  is given by the half--quadratic approximation $a_{ij} x_{ij}^2 + b_{ij} x_{ij} \approx a_{ij} x_{ij}^2 + 2 b_{ij} ({{x_{ij}^2 + \beta^2} })^{\frac{1}{2}}$, with $\beta \approx 0$; thus, $ \omega_{ij} = a_{ij} +  \frac{1}{2}{ b_{ij} }/({{x_{ij}^2 + \beta^2} })^{\frac{1}{2}}$.  This approximation is presented  with the sole aim of illustrating the potential of our approach. It is clear that the dual algorithm derived according to Proposition \ref{prop:qp} is more accurate, faster and requires less memory to be implemented. 
 
\begin{figure}[t]
\centerline{
\includegraphics[width=\linewidth]{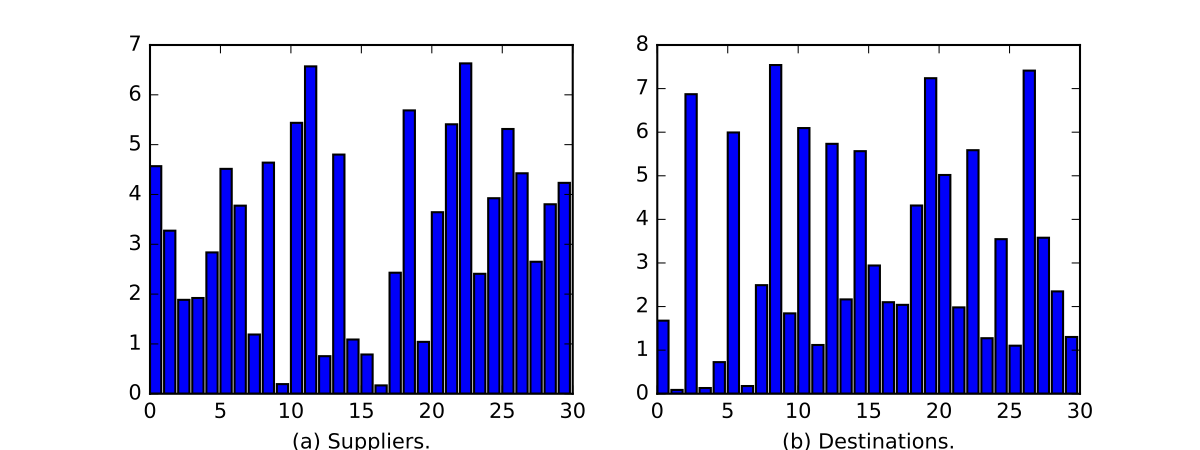}
}
\caption{Randomly generated levels of production, $p$, and demand, $q$.}
\label{fig:fig2}
\end{figure}

\begin{figure}[t]
\centerline{
\includegraphics[width=\linewidth]{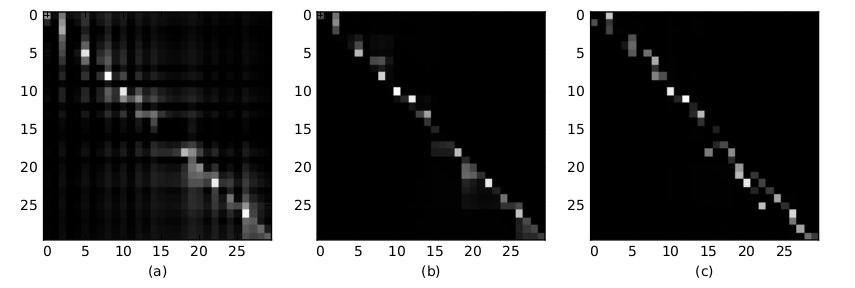}
}
\caption{Computed transport values $x$ (in grayscale) corresponding to (a) quadratic cost, (b) half--quadratic $L_1$--approximation and (c) half--quadratic $L_0$--approximation. The solution computed with the non--convex cost function ($L_0$--approximation) depicted in panel (c) is sparser.}
\label{fig:fig3}
\end{figure}


\section{Discussion and Conclusions}
The transportation problem is the base of the Earth Mover Distance which has become a relevant metric to for compare distributions in applications to data analysis and computer vision. 
The presented technique can motivate the design of new algorithms in those areas. 

In order to demonstrate the versatility of our proposal, we generate two random vectors $p$ and $q$ (depicted in Figure \ref{fig:fig2}); we compute the optimum transported volumes $x$ with three cost function models: the quadratic ($c_{ij}x_{ij}^2$), the approximation $L_1$ ($c_{ij}\sqrt{x_{ij}^2+\beta^2}$, with $\beta^2= 1\times10^{-3}$) and the approximation $L_0$ ($c_{ij}x_{ij}^2 / (\beta^2 + x_{ij}^2)$, with $\beta^2= 1\times10^{-1}$). In all the cases, we use $c_{ij} = |i-j|+1$. Figure \ref{fig:fig3} depicts the computed $x$ values. One can observe that the quadratic cost function promotes dense solutions; \emph{i.e.}, there are many $x$'s with small values. On the other hand, one can observe the sparseness of the solution is induced with the use of the approximated $L_1$--norm. Such sparsity is emphasised with the approximated $L_0$--norm. 

 We have presented a model to approximate solutions to the general transportation problems by approximating the transportation cost functions with half--quadratic functions. The approach guarantees convergence using an alternated minimisation scheme. In the case of a non--convex cost function $f$ the convergence is guaranteed to a local minimum. Although we present a minimisation algorithm with reduced memory requirements, our scheme accepts other efficient solvers for the quadratic transportation subproblem; such as those reported in Refs. \cite{adlakha13,megiddo93, cosares94}.
 










\end{document}